\documentclass[ECP,preprint]{ejpecp}

\usepackage[T1]{fontenc}
\usepackage[utf8]{inputenc}
\usepackage{mleftright}
\usepackage{xcolor}
\usepackage[scr = rsfs]{mathalpha}
\usepackage{esint}
\usepackage{zref-clever}
\zcsetup{
    abbrev=true, %
    cap=true,
    nameinlink=false,
	rangetopair=false,
	endrange=stripprefix,
	sort=false,
	reffont=\upshape,
	lastsep={, and },
	comp=false
    }

\zcRefTypeSetup{equation}{
Name-sg = ,
name-sg = ,
Name-pl = ,
name-pl = ,
refbounds = {\textup{(},,,\textup{)}},
+refbounds-rb = {\textup{(},,,},
+refbounds-re = {,,,\textup{)}},
rangesep = {--}
}
\zcRefTypeSetup{thmstepnvi}{
  Name-sg = Step,
  name-sg = step,
  Name-pl = Steps,
  name-pl = steps
}
\zcRefTypeSetup{thmstepnvii}{
  Name-sg = Step,
  name-sg = step,
  Name-pl = Steps,
  name-pl = steps
}
\zcsetup{countertype={theorem=theorem}}
\AddToHook{env/assumptions/begin}{\zcsetup{countertype={theorem=assumptions}}}
\AddToHook{env/assumption/begin}{\zcsetup{countertype={theorem=assumption}}}
\AddToHook{env/claim/begin}{\zcsetup{countertype={theorem=claim}}}
\AddToHook{env/condition/begin}{\zcsetup{countertype={theorem=condition}}}
\AddToHook{env/conjecture/begin}{\zcsetup{countertype={theorem=conjecture}}}
\AddToHook{env/corollary/begin}{\zcsetup{countertype={theorem=corollary}}}
\AddToHook{env/definitions/begin}{\zcsetup{countertype={theorem=definitions}}}
\AddToHook{env/definition/begin}{\zcsetup{countertype={theorem=definition}}}
\AddToHook{env/facts/begin}{\zcsetup{countertype={theorem=facts}}}
\AddToHook{env/fact/begin}{\zcsetup{countertype={theorem=fact}}}
\AddToHook{env/heuristics/begin}{\zcsetup{countertype={theorem=heuristics}}}
\AddToHook{env/hypothesis/begin}{\zcsetup{countertype={theorem=hypothesis}}}
\AddToHook{env/hypotheses/begin}{\zcsetup{countertype={theorem=hypotheses}}}
\AddToHook{env/lemma/begin}{\zcsetup{countertype={theorem=lemma}}}
\AddToHook{env/notations/begin}{\zcsetup{countertype={theorem=notations}}}
\AddToHook{env/notation/begin}{\zcsetup{countertype={theorem=notation}}}
\AddToHook{env/proposition/begin}{\zcsetup{countertype={theorem=proposition}}}
\AddToHook{env/example/begin}{\zcsetup{countertype={theorem=example}}}
\AddToHook{env/exercise/begin}{\zcsetup{countertype={theorem=exercise}}}
\AddToHook{env/problem/begin}{\zcsetup{countertype={theorem=problem}}}
\AddToHook{env/question/begin}{\zcsetup{countertype={theorem=question}}}
\AddToHook{env/remark/begin}{\zcsetup{countertype={theorem=remark}}}
\AddToHook{env/proposition/begin}{\zcsetup{countertype={theorem=proposition}}}
\newcommand{\cref}[1]{\text{\zcref{#1}}} %

\SHORTTITLE{Edwards--Wilkinson fluctuations in subcritical stochastic heat equations}

\TITLE{Edwards--Wilkinson fluctuations in subcritical 2D stochastic heat equations}

\AUTHORS{%
  Alexander~Dunlap\footnote{Duke University, USA.
    \EMAIL{alexander.dunlap@duke.edu}}
  \and
  Cole~Graham\footnote{University of Wisconsin--Madison, USA.
    \EMAIL{graham@math.wisc.edu}}}


\AMSSUBJ{60F05}
\AMSSUBJSECONDARY{60H15; 35R60; 82B44}

\SUBMITTED{June 5, 2024}

\ARXIVID{2405.09520}

\ABSTRACT{%
  We study 2D nonlinear stochastic heat equations under a logarithmically attenuated white-noise limit with subcritical coupling.
  We show that solutions asymptotically exhibit Edwards--Wilkinson fluctuations.
  This extends work of Ran Tao, which required a stricter condition on the coupling.
  Part of the limiting fluctuation is measurable with respect to the original noise and the remainder is independent.
  \add{We also show that these statistics are universal in the sense that they are independent of the fine details of the model.}
}

\newcommand*{\op}[1]{\operatorname{#1}}
\newcommand*{\m}[1]{\mathcal{#1}}
\newcommand*{\R}{\mathbb{R}}
\newcommand*{\N}{\mathbb{N}}
\newcommand*{\E}{\mathbb{E}}

\newcommand*{\D}{\mathbb{D}}

\newcommand*{\dn}{\mathrm{d}}

\newcommand*{\ds}{\,\mathrm{d}}
\renewcommand*{\d}{\;\mathrm{d}}
\newcommand*{\abs}[1]{\left |#1 \right|}
\newcommand*{\abss}[1]{|#1|}
\newcommand*{\absb}[1]{\big|#1\big|}

\newcommand*{\anon}{\,\cdot\,}

\newcommand*{\h}[1]{\hat{#1}}

\let\bar\smallbar

\renewcommand*{\And}{\quad \text{and} \quad}

\newcommand*{\Lip}{\op{Lip}}
\newcommand*{\Var}{\op{Var}}
\newcommand*{\Cov}{\op{Cov}}
\newcommand*{\Jbar}{\bar{J}}
\newcommand*{\Jtilde}{\widetilde{J}}
\newcommand{\add}[1]{{#1}}

\begin{document}

\section{Introduction}

We study the vector-valued semilinear stochastic heat equation
\begin{equation}
  \label{eq:SHE}
  \dn \add{u}_{t}(x)=\frac{1}{2}\Delta \add{u}_{t}(x)\dn t+\gamma_{\rho} \sigma(\add{u}_{t}) \add{\ds W_{t}^\rho}(x),\qquad t\in\R,x\in\R^{2}
\end{equation}
indexed by a small parameter $0<\rho\ll1$.
The solution $\add{u}_t(x) = \add{u}_t^\rho(x)$ takes values in $\R^{m}$ for some $m\in\N$.
\add{To balance clarity and brevity, we write $\add{u}^\rho$ in result statements, but generally suppress the dependence elsewhere.}

The nonlinearity $\sigma\colon\R^{m}\to\m{H}_{+}^{m}$ is Lipschitz with respect to the Frobenius norm $\abs{\anon}_{\mathrm{F}}$ on the space $\m{H}_{+}^{m}\subset (\R^{m})^{\otimes2}$ of nonnegative-definite symmetric $m\times m$ matrices.
The stochastic forcing $\dn W^\rho$ is \add{a spatial regularization of} an $m$-dimensional spacetime white noise $\dn W$, itself a vector of $m$ i.i.d.\ spacetime white noises.
Let $\m{G}$ denote the heat semigroup:
\begin{equation*}
  \m{G}_{s}f(x) = (G_{s}*f)(x)\qquad\text{for}\qquad G_{s}(x) \coloneqq \frac{1}{2\pi s}\exp\left(-\frac{|x|^{2}}{2s}\right),
\end{equation*}
where $*$ denotes spatial convolution.
\add{Then $\dn W_t^\rho(x) \coloneqq \m{G}_\rho \dn W_t (x) = (G_\rho \ast \dn W_t)(x).$}
We logarithmically attenuate this noise through the factor
\begin{equation*}
  \gamma_{\rho}\coloneqq\left[\frac{4\pi}{\log(1+1/\rho)}\right]^{1/2}.
\end{equation*}

Stochastic heat equations in $2+1$ dimensions have attracted considerable interest due to their critical scaling properties.
This line of inquiry includes both linear \cite{BC98,CSZ17,CSZ21} and semilinear \cite{DG22,Tao22,DG23b} models, and we refer to the introduction of our recent work \cite{DG23b} for further background.

In this note, we explore the fluctuations of \eqref{eq:SHE} in the white-noise limit $\rho \to 0$.
We show that suitably rescaled fluctuations of solutions to \cref{eq:SHE} converge in law to solutions to the Edwards--Wilkinson equation (the additive stochastic heat equation), and in particular are asymptotically Gaussian.
Ran Tao \cite{Tao22} has shown the same for equations satisfying a certain bound on $\Lip(\sigma)$.
The present work relaxes this restriction and clarifies the role of the noise in the limiting process.

In recent work \cite{DG23b}, we considered instead the pointwise statistics of $v$ as $\rho\to0$.
These are generally non-Gaussian and are governed by the forward-backward stochastic differential equation (FBSDE)
\begin{equation}
  \label{eq:FBSDE}
  \dn Z_{q} = \E[\sigma^{2}(Z_{1}) \mid Z_{q}]^{1/2}\dn B_{q}
\end{equation}
introduced in \cite{DG22}, which evolves over scales $q\in[0,1]$.
Here $(\anon)^{1/2}$ denotes the positive-semidefinite matrix square root.
The scale parameter $q$ is related to the time parameter $t$ in \cref{eq:SHE} by $t\approx T-\rho^{q}$ for a final time $T$ of interest.
The FBSDE \cref{eq:FBSDE} is essentially equivalent to the renormalization flow
\begin{equation}
  \label{eq:renorm-flow}
  \partial_{q}H_{q}(b)=\frac{1}{2}[H_{q}(b):\nabla_{b}^{2}]H_{q}(b) = \frac{1}{2} \sum_{k,\ell = 1}^m (H_q)_{k \ell}(b) \frac{\partial^2H_q}{\partial b_k \partial b_\ell}(b),\quad H_{0}(b)=\sigma^{2}(b),
\end{equation}
the deterministic quasilinear heat equation satisfied by the decoupling function
\begin{equation*}
  H_q(b)\coloneqq\E[\sigma^{2}(Z_{1})\mid Z_{1-q}=b].
\end{equation*}
In \cite{DG23b}, $Q_{\mathrm{FBSDE}}(\sigma)$ denotes the maximal $q>0$ up to which $\Lip \sqrt{H_q} < \infty$.
\begin{definition}
  We say $\sigma$ is \emph{\add{$L^2$-}subcritical} if $Q_{\mathrm{FBSDE}}(\sigma) > 1$ and
  \begin{equation*}
    \limsup_{\abs{w} \to \infty} \frac{|\sigma(w)|_{\mathrm{F}}}{\abs{w}} < 1.
  \end{equation*}
\end{definition}
If $\sigma$ is \add{$L^2$-}subcritical, then we can solve \cref{eq:FBSDE} starting at the macroscopic scale $q = 0$ and proceed down to the microscopic scale $q = 1$.
This allows us to resolve all scales using the renormalization flow and thus understand the pointwise statistics of solutions of \eqref{eq:SHE}.
In \cite{DG23b}, we showed that $\sigma$ is \add{$L^2$}-subcritical if $\Lip \sigma <1$.
When $m=1$, we proved less stringent sufficient conditions for subcriticality in \cite{DG23a}.
In contrast, in our normalization, the proof of Edwards--Wilkinson fluctuations in \cite{Tao22} requires the stronger condition $\Lip(\sigma)<\frac1{2\sqrt{6}}<1$.
We note that some subcriticality assumption is necessary, as the linear problem $\sigma(u)=\beta u$ exhibits a phase transition at $\beta= 1$ \cite{BC98,CSZ19,GQT21,CSZ21,Zyg24}.

In the present work, we are concerned with the macroscopic rather than pointwise statistics of the solution $\add{u}$.
To observe nontrivial fluctuations in $\add{u}$, we subtract the mean $\overline{\add{u}}_{t}(x)\coloneqq \m{G}_{t}\add{u}_{0}(x)$ from $\add{u}_{t}$ and rescale by the large factor $\gamma_{\rho}^{-1}$.
To express the limit, we introduce the functions
\begin{equation}
  \Jbar_{1}(b) \coloneqq \E[\sigma(Z_{1}) \mid Z_{0}=b] \And \Jtilde_{1}(b) \coloneqq \Var[\sigma(Z_1) \mid Z_0 = b]^{1/2}.\label{eq:Jbardef}
\end{equation}
We show that the limiting fluctuations of $\add{u}$ satisfy the Edwards--Wilkinson equation
\begin{equation}
  \label{eq:EW-twonoises}
  \dn\m{U}_{t}(x) = \frac{1}{2}\Delta\m{U}_{t}(x)\dn t+(\Jbar_{1} \circ \overline{\add{u}}_{t})(x) \dn W_{t}(x) + (\Jtilde_{1} \circ \overline{\add{u}}_{t})(x) \dn\widetilde{W}_{t}(x), \quad \m{U}_{0} \equiv 0.
\end{equation}
Here, $\dn\widetilde{W}$ is an $m$-dimensional spacetime white noise independent of (and identically distributed to) $\dn W$.
By linearity, we can write $\m{U}_{t}=\overline{\m{U}}_{t}+\widetilde{\m{U}}_{t}$, where
\begin{equation}
  \label{eq:EW-twonoises-1}
  \dn\overline{\m{U}}_{t} =\frac{1}{2}\Delta\overline{\m{U}}_{t}\dn t + (\Jbar_{1} \circ \overline{\add{u}}_{t}) \dn W_{t}, \qquad \dn\widetilde{\m{U}}_{t} =\frac{1}{2}\Delta\widetilde{\m{U}}_{t}\dn t + (\Jtilde_{1} \circ \bar{\add{u}}_{t}) \dn \widetilde W_{t},
\end{equation}
and $\overline{\m{U}}_{0}=\widetilde{\m{U}}_{0}\equiv0$.
Moreover, solutions of \cref{eq:EW-twonoises} agree in law with those of
\begin{equation*}
  \dn\m{V}_{t} = \frac{1}{2} \Delta\m{V}_{t} \dn t + (J_{1} \circ \overline{\add{u}}_{t}) \dn W_{t}, \quad \m{V}_0 \equiv 0,
\end{equation*}
where we define $J_{1}(b) \coloneqq H_{1}(b)^{1/2} = [\Jbar_{1}(b)^{2} + \Jtilde_{1}(b)^{2}]^{1/2}.$
The separation of $\dn W$ and $\dn\widetilde{W}$ in \cref{eq:EW-twonoises} allows us to state a more precise form of the convergence.
\begin{theorem}
  \label{thm:maintheorem}
  Let $\sigma$ be \add{$L^2$-}subcritical and $\add{u}_{0}\in L^{\infty}(\R^{2})$ be deterministic.
  Then the pair of random distributions $(\gamma_{\rho}^{-1}(\add{u^\rho}-\overline{\add{u}}),\overline{\m{U}})$ converges jointly in law, in the topology of distributions on $\R_{+}\times\R^{2}$, to $(\m{U},\overline{\m{U}})$ as $\rho\to0$.
  Precisely, for any test functions $\psi,\phi\in\m{C}_{\mathrm{c}}^{\infty}(\R_{+}\times\R^{2})$, as $\rho\to0$ we have the joint convergence in law
  \begin{equation}\label{eq:mainthm-convergence}
    \begin{aligned}
      & \left(\gamma_{\rho}^{-1}\int_{\R_{+}\times\R^{2}}\psi_{t}(x)[\add{u}_{t}^{\add{\rho}}(x)-\overline{\add{u}}_{t}(x)]\d x\ds t,\int_{\R_{+}\times\R^{2}}\phi_{t}(x)\overline{\m{U}}_{t}(x)\d x\ds t\right)\\
      & \qquad\Longrightarrow\left(\int_{\R_{+}\times\R^{2}}\psi_{t}(x)\m{U}_{t}(x)\d x\ds t,\int_{\R_{+}\times\R^{2}}\phi_{t}(x)\overline{\m{U}}_{t}(x)\d x\ds t\right).
    \end{aligned}
  \end{equation}
\end{theorem}
The emergence of Edwards--Wilkinson fluctuations is essentially equivalent to the concentration of macroscopic averages of $\sigma^{2}\circ \add{u}$ and $\sigma\circ \add{u}$.
This is a consequence of the martingale central limit theorem, which has been previously used to study fluctuations in stochastic heat equations~\add{\cite{CNN22}}\cite{CN21,NN23}.
Other approaches for proving Edwards--Wilkinson fluctuations in stochastic PDEs include Markov chain approximations \cite{GRZ18,Kot24}, the Malliavin calculus and the second-order Poincaré inequality \cite{GL20,DGRZ20,Gu20,Tao22,GHL23}, chaos expansions and the fourth moment method \cite{CSZ17,CSZ20,LZ22,CC22,CCR23}, and cluster expansions \cite{MU18}.

The large-scale concentration of $\sigma^{2}\circ \add{u}$ and $\sigma\circ \add{u}$ and the correlation structure in \cref{eq:mainthm-convergence} have an intuitive explanation.
We showed in \cite{DG23b} that the random field $\add{u}$ decorrelates rapidly in space as $\rho \to 0$.
One can thus view the noise in \cref{eq:SHE} as the mean of $\sigma(\add{u}_t)$ multiplying macroscopic modes of $\dn W_t$, plus the rapid fluctuations of $\sigma(\add{u}_t)$ multiplying microscopic modes of $\dn W_t$.
The former contributes $(\Jbar_1\circ \overline{\add{u}}_t) \dn W_t$ in \cref{eq:EW-twonoises-1}.
The latter involves only high-frequency modes of $\dn W_t$, which become independent of $\dn W_t$ in the limit $\rho \to 0$.
This explains the independent noise $(\Jtilde_1\circ \overline{\add{u}}_t) \dn \widetilde W_t$ in \cref{eq:EW-twonoises-1}.

\add{Our vector-valued analysis enables another approach to \cref{thm:maintheorem}.
  Let $\h u$ denote the vector $(u, w)^\top \in \R^{2m}$, where $u$ solves \cref{eq:SHE} and $w$ solves the $\gamma_\rho$-attenuated Edwards--Wilkinson equation driven by the same noise $\dn W^\rho$.
  Then $\h u$ also solves an equation of the form \cref{eq:SHE} with an asymmetric block matrix $\h \sigma$.
  Symmetrizing and solving \cref{eq:renorm-flow}, the resulting block-matrix $\h H$ will contain $H$ on the diagonal and $\Jbar$ on the off-diagonal.
  Thus the macroscopic statistics of $\h u$ encompass both those of $u$ and the joint relationship between $u$ and the Edwards--Wilkinson equation detailed in \cref{eq:mainthm-convergence}.
  In this sense, there is no need for the joint statement in \cref{thm:maintheorem}, as all its information can be read from the statistics of $\h u$ alone.
  Nonetheless, our argument easily yields the joint statistics in \cref{thm:maintheorem}, so we elect to prove them directly.
}

\add{%
  We also study the universality of the Edwards--Wilkinson statistics in \cref{thm:maintheorem}.
  In \cite{DG23b}, we showed that the pointwise statistics of \cref{eq:SHE} are insensitive to the fine details of the model.
  The same holds at the macroscopic level.
  To formulate our universality result, we study approximate mild solutions of a closely related model.
  Given a random field $v_t^\rho(x),$ define the mild operator
  \begin{equation}
    \label{eq:mild}
    (\m{T}^\rho v)_t(x) \coloneqq \m{G}_t v_0(x) + \gamma_\rho \int_0^t \m{G}_{t - s + \rho}[\sigma(v_s) \ds W_s](x).
  \end{equation}
    Fixed points of $\m{T}^\rho$ solve a variation of \cref{eq:SHE} with noise $\gamma_\rho \m{G}_\rho[\sigma(v)\ds W]$ rather than $\gamma_\rho \sigma(v) \m{G}_\rho \dn W$.
  We find this ``post-smoothed'' noise to be more convenient, and the choice has little effect on the large-scale behavior of $v$.
}

\add{%
  The difference $\m{T}^\rho v - v$ expresses the extent to which a field $v$ fails to solve the post-smoothed equation.
  We show that the statistics in \cref{eq:mainthm-convergence} hold when this difference is small.
  Precisely, we say $v^\rho$ is an ``approximate mild solution'' if it satisfies the following three conditions.
  For some $\ell > 2$ and all $T > 0$,
  \begin{equation}
    \label{eq:moment}
    \tag{H1}
    \sup_{\rho \in (0, 1], \, t \in [0, T], \, x \in \R^2} \E \abs{v_t^\rho(x)}^\ell < \infty.
  \end{equation}
  For all $T > 0$,
  \begin{equation}
    \label{eq:pointwise-small}
    \tag{H2}
    \lim_{\rho \to 0} \sup_{[0,T] \times \R^2} \E \abs{\m{T}^\rho v^\rho - v^\rho}^2 = 0.
  \end{equation}
  And for all $\psi \in \m{C}_{\mathrm{c}}^\infty(\R_+ \times \R^2)$,
  \begin{equation}
    \label{eq:macro-small}
    \tag{H3}
    \lim_{\rho \to 0} \E \Big|\gamma_\rho^{-1}\int_{\R_+ \times \R^2} \psi_t(x) (\m{T}^\rho v^\rho - v^\rho)_t(x) \d x \ds t \Big|^2 = 0.
  \end{equation}
  In \cite{DG23b}, we show that \zcref[range]{eq:moment,eq:pointwise-small} imply that $v$ has the proper pointwise behavior.
  The final condition \cref{eq:macro-small} ensures that $v$ is nearly a solution in a macroscopic sense.%
}
\add{%
  \begin{theorem}
    \label{thm:universal}
    Let $\sigma$ be $L^2$-subcritical.
    Suppose a random field $\bigl(v_t^\rho(x)\bigr)_{t \geq 0, x \in \R^2}$ is an approximate mild solution in the sense of \zcref[range]{eq:moment,eq:macro-small}.
    Then \cref{eq:mainthm-convergence} holds for $v$ as well.
  \end{theorem}
  In~\cite{DG23b}, we show that the solution $u$ of \cref{eq:SHE} satisfies \zcref[range]{eq:moment,eq:pointwise-small}.
  We show \cref{eq:macro-small} in \cref{prop:macroscopic-mild} below, so in fact \cref{thm:maintheorem} is a special case of \cref{thm:universal}.
}

\add{
  Because our hypotheses on $v$ do not involve any derivatives, \zcref[range]{eq:moment,eq:macro-small} is a rather expansive notion of ``approximate solution'' that should include many small-scale variations on \cref{eq:SHE}.
  For example, we expect that it applies to certain discretizations of \cref{eq:SHE}.
  Thus the statistics in \cref{eq:mainthm-convergence} are universal in that they hold for a large class of models that differ at the fine ($\rho$) scale.%
}

\section{Concentration}

\add{%
  Let $v$ be an approximate mild solution in the sense of \zcref[range]{eq:moment,eq:macro-small}.%
} %
We prove concentration for macroscopic averages of $\sigma^{2}\circ v_{t}$ and $\sigma\circ v_{t}$, beginning with a simple convolution estimate.
Given $\xi>0$ and $x\in\R^{2}$, let $\square_{\xi}(x) \coloneqq [-\xi/2,\xi/2]^{2}+x$ and $S_{\xi}(x) \coloneqq \xi^{-2}\mathbf{1}_{\square_{\xi}(0)}(x)$.
\begin{lemma}
  \label{lem:squaregaussiansame}
  There exists a constant $C < \infty$ such that for all $s>0$ and $\xi\in(0,\sqrt{s}\,]$,
  \begin{equation}
    \|G_{s}*S_{\xi}-G_{s}\|_{L^{1}}\le C\xi^{2}/s.\label{eq:squaregaussiansame}
  \end{equation}
\end{lemma}
\noindent
\begin{proof}
  In the following, we allow the constant $C$ to increase from instance to instance.
  The symmetry of $S_{\xi}$ yields
  \begin{equation*}
    (G_{s}*S_{\xi}-G_{s})(x)=\fint_{\square_{\xi}(x)}[G_{s}(y)-G_{s}(x)-\nabla G_{s}(x)\cdot(y-x)]\d y.
  \end{equation*}
  By Taylor's theorem, we have
  \begin{equation}
    \label{eq:applyTaylor}
    |G_{s}*S_{\xi}-G_{s}|(x)\le C\xi^{2}\sup_{\square_{\xi}(x)}|\D^{2}G_{s}|_{\mathrm{F}}.
  \end{equation}
  A calculation then yields
  \begin{equation*}
    |\D^{2}G_{s}(x)|_{\mathrm{F}}\le Cs^{-1}(x^{2}/s + 1) G_{s}(x)\le Cs^{-1}G_{s/2}(x).
  \end{equation*}
  Since $\xi^{2}\le s$, the supremum over $\square_{\xi}(x)$ does not significantly alter this bound, and we have
  \begin{equation*}
    \sup_{\square_{\xi}(x)}|\D^{2}G_{s}|_{\mathrm{F}}\le Cs^{-1}G_{s/2}(x).
  \end{equation*}
  Using this in \cref{eq:applyTaylor} and integrating, we obtain \cref{eq:squaregaussiansame}.
\end{proof}
We now prove the concentration of macroscopic averages of $\sigma^{2}\circ v_{t}$.
\begin{proposition}
  For any fixed $s>0$, $z_{\rho}\to z\in\R^{2}$, and $r_{\rho}\to r>0$ (as $\rho\downarrow0$),
  \begin{equation}
    \label{eq:sigma2concentration}
    \lim_{\rho\downarrow0} \E\left|\left[\m{G}_{r_{\rho}}(\sigma^{2}\circ v_{s}^{\add{\rho}})(z_{\rho})\right]^{1/2}-\left[\m{G}_{r}(H_{1}\circ\overline{v}_{s})(z)\right]^{1/2}\right|_{\mathrm{F}}^{2} = 0.
  \end{equation}
  As a consequence, we have
  \begin{equation}
    \label{eq:sigma2concinprb}
    \m{G}_{r_{\rho}}(\sigma^{2}\circ v_{s}^{\add{\rho}})(z_{\rho})\to\m{G}_{r}(H_{1}\circ\overline{v}_{s})(z)\ \text{in probability as }\rho\downarrow0.
  \end{equation}
\end{proposition}
\begin{proof}
  It is standard that \cref{eq:sigma2concentration} implies \cref{eq:sigma2concinprb}, so we just need to prove \cref{eq:sigma2concentration}.
  We rely on results of \cite{DG23b}, so we introduce further notation from that work.
  We define a regularized logarithm $\mathsf{L}(\tau)=\log(\tau+1)$ as well as
  \begin{equation*}
    \mathsf{S}_{\rho}(\tau)\coloneqq\frac{\mathsf{L}(\tau/\rho)}{\mathsf{L}(1/\rho)} \And \mathsf{T}_{\rho}(q) \coloneqq \rho[(1/\rho+1)^{q}-1].
  \end{equation*}
  Then $\mathsf{S}_{\rho}$ and $\mathsf{T}_{\rho}$ are inverse functions representing an exponent and time, respectively.

  Theorem\add{s}~6.3 \add{and 8.5} of \cite{DG23b} give us a function $\omega\colon(0,1]\to(0,\infty)$ such that, as $\rho\to 0$, we have \add{$\omega(\rho)\to \infty$} but $\omega(\rho)=(1/\rho)^{o(1)}$ and the following holds.
  If we define $q_{\rho}\coloneqq\mathsf{S}_{\rho}\big(s/\omega(\rho)\big)$, so $\omega(\rho)\mathsf{T}_{\rho}(q_{\rho})=s$, then 
  \begin{equation}
    \label{eq:firstlimit}
    \adjustlimits\lim_{\rho\downarrow0}\sup_{x\in\R^{2}}\E\left|\big[S_{\mathsf{T}_{\rho}(q_{\rho})^{1/2}}*(\sigma^{2}\circ v_{s})\big](x)^{1/2}-H_{q_{\rho}}\big(\m{G}_{\mathsf{T}_{\rho}(q_{\rho})}v_{s}(x)\big)^{1/2}\right|_{\mathrm{F}}^{2}=0.
  \end{equation}
  \add{%
    This estimate is the centerpiece of \cite{DG23b}, and proves essential in this work as well.
    In more detail, \cite[Theorem~6.3]{DG23b} establishes \cref{eq:firstlimit} for fixed points of $\m{T}$ (exact mild solutions), and \cite[Theorem~8.5]{DG23b} extends it by continuity to all approximate mild solutions.%
    }

    We next apply \cite[Proposition A.2]{DG23b} to a random spatial value drawn from $\m{N}(0,r_\rho)$, so the expectation in the proposition corresponds to spatial convolution:
  \begin{align}
    \label{eq:secondlimit}
    \begin{aligned}
      & \E\left|\big[G_{r_{\rho}}*S_{\mathsf{T}_{\rho}(q_{\rho})^{1/2}}*(\sigma^{2}\circ v_{s})\big](z_\rho)^{1/2}-\m{G}_{r_{\rho}}(H_{q_{\rho}}\circ\m{G}_{\mathsf{T}_{\rho}(q_{\rho})}v_{s})(z_\rho)^{1/2}\right|_{\mathrm{F}}^{2}\\
      &\hspace{1cm}\le \int_{\R^2} G_{r_{\rho}}(z_\rho-x)\E\Big|\big[S_{\mathsf{T}_{\rho}(q_{\rho})^{1/2}}*(\sigma^{2}\circ v_{s})\big](x)^{1/2}\\
      &\hspace{6cm}-\big[H_{q_{\rho}}\big(\m{G}_{\mathsf{T}_{\rho}(q_{\rho})}v_{s}(x)\big)\big]^{1/2}\Big|_{\mathrm{F}}^{2}\d x \to 0
    \end{aligned}
  \end{align}
  as $\rho \to 0$ by \cref{eq:firstlimit}.
  Writing $|\cdot|_{*}$ for the nuclear norm, the Powers--Størmer inequality yields
  \begin{align*}
    \E \Big|\big[G_{r_{\rho}}*S_{\mathsf{T}_{\rho}(q_{\rho})^{1/2}}*(\sigma^{2}\circ &v_{s})\big](z_\rho)^{1/2}- \m{G}_{r_{\rho}}(\sigma^{2}\circ v_{s})(z_\rho)^{1/2}\Big|_{\add{\mathrm{F}}}^{2}\\
                                                                                           & \le \E\big|\big(G_{r_{\rho}}*S_{\mathsf{T}_{\rho}(q_{\rho})^{1/2}}-G_{r_{\rho}}\big)*(\sigma^{2}\circ v_{s})(z_\rho)\big|_{*}\\
                                                                                           & \le\|G_{r_{\rho}}*S_{\mathsf{T}_{\rho}(q_{\rho})^{1/2}}-G_{r_{\rho}}\|_{L^{1}}\sup_{x\in\R^2}\E|\sigma^{2}(v_s(x))|_{*}.
  \end{align*}
  Since $\mathsf{T}_{\rho}(q_{\rho})=s/\omega(\rho)\to0$ and $r_{\rho}\to r>0$ as $\rho\to0$, \cref{lem:squaregaussiansame} implies that the first factor goes to $0$ as $\rho\to 0$.
  For the second factor,
  \begin{equation*}
    \add{\sup_{x\in\R^2}}\E|(\sigma^{2} \circ v_{s})(x)|_{*}\le \add{\sup_{x\in\R^2}}\E|(\sigma \circ v_{s})(x)|_{\mathrm{F}}^{2}.
  \end{equation*}
  \add{Because $\sigma$ is Lipschitz,} the right side is bounded by \cref{eq:moment} and Jensen.

  It remains to show that $\m{G}_{r_{\rho}}(H_{q_{\rho}}\circ\m{G}_{\mathsf{T}_{\rho}(q_{\rho})}v_{s})(z_\rho) \to \m{G}_{r}(H_{1}\circ\overline{v}_{s})(z)$.
  To see this, we note that the properties of $\omega$ imply that $q_\rho \to 1$ as $\rho \to 0$.
  Thus by \add{Theorems~1.7 and 8.5} of \cite{DG23b}, $\E \abss{\m{G}_{\mathsf{T}_{\rho}(q_{\rho})}v_{s}(x) - \bar{v}_s(x)}^2 \to 0$ for all $x \in \R^2$.
  Also, the root-decoupling function $H^{1/2}$ is Lipschitz in space by subcriticality and 1/2-H\"older in time by \cite[Proposition 2.14]{DG23b}.
  Thus the moment bound in \cref{eq:moment} yields
  \begin{equation}
    \label{eq:H1toaverage}
    \E\left|H_{q_{\rho}}\big(\m{G}_{\mathsf{T}_{\rho}(q_{\rho})}v_{s}(x)\big)^{1/2}-(H_{1} \circ \bar{v}_s)(x)^{1/2}\right|_{\mathrm{F}}^{2}\to0\qquad\text{as }\rho\downarrow0
  \end{equation}
  for each $x \in \R^2$.
  Applying \cite[Proposition A.2]{DG23b} as in \cref{eq:secondlimit}, we find
  \begin{align}
    \E \Big|\m{G}_{r_{\rho}}(&H_{q_{\rho}}\circ\m{G}_{\mathsf{T}_{\rho}(q_{\rho})}v_{s})(z_\rho)^{1/2} - \m{G}_{r_\rho}(H_1 \circ \bar{v}_s)(z_\rho)^{1/2}\Big|_{\mathrm{F}}^2\label{eq:moving-Gaussians}\\
                             &\leq \int_{\R^2} G_{r_\rho}(z_\rho - x) \E \absb{H_{q_{\rho}}\big(\m{G}_{\mathsf{T}_{\rho}(q_{\rho})}v_{s}(x)\big)^{1/2}-(H_{1} \circ \bar{v}_s)(x)^{1/2}}_{\mathrm{F}}^2 \d x.\nonumber
  \end{align}
  The first factor is bounded by a $\rho$-independent envelope in $L^1(\R^2)$.
  The second factor tends to zero pointwise by \cref{eq:H1toaverage} and is uniformly bounded by \cref{eq:moment}.
  Thus by dominated convergence, the expression in \cref{eq:moving-Gaussians} vanishes as $\rho \to 0$.
  Finally, the heat flow is spacetime continuous, so $\m{G}_{r_\rho}(H_1 \circ \bar{v}_s)(z_\rho) \to \m{G}_{r}(H_1 \circ \bar{v}_s)(z)$ as $\rho \to 0$.
  This completes the proof of the proposition.
\end{proof}
We next treat $\sigma\circ v_{s}$.
We recall the definition \cref{eq:Jbardef} of $\Jbar_{1}$.
\begin{proposition}
  For each fixed $s>0$, $z_{\rho}\to z \in \R^{2}$, $r_{\rho}\to r>0$ (as $\rho\downarrow0$), and bounded function $f \colon \R^2 \to (\R^m)^{\otimes 2}$, we have
  \begin{equation}
    \lim_{\rho\downarrow0} \E\left|\m{G}_{r_{\rho}}[(\sigma\circ v_{s}^{\add{\rho}})f](z_{\rho})-\m{G}_{r}[(\Jbar_{1}\circ\overline{v}_{s})f](z)\right|_{\mathrm{F}}^{2}\to0.\label{eq:sigmaconc}
  \end{equation}
  As a consequence, we have
  \begin{equation}
    \m{G}_{r_{\rho}}[(\sigma\circ v_{s}^{\add{\rho}})f](z_{\rho})\to\m{G}_{r}[(\Jbar_{1}\circ\overline{v}_{s})f](z)\text{ in probability as }\rho\downarrow0.\label{eq:sigmaconcinprob}
  \end{equation}
\end{proposition}
\begin{proof}
  As above, \cref{eq:sigmaconc} implies \cref{eq:sigmaconcinprob}, so we just prove \cref{eq:sigmaconc}.
  Using uniform integrability from \cref{eq:moment}, \add{Theorems~1.4 and 8.5} of \cite{DG23b} imply that
  \begin{equation*}
    \E \, \m{G}_{r_{\rho}}[(\sigma\circ v_{s})f](z_{\rho}) \to \m{G}_{r}[(\Jbar_{1}\circ\overline{v}_{s})f](z) \qquad\text{as }\rho\downarrow0.
  \end{equation*}
  It thus suffices to show that $\m{G}_{r_{\rho}}[(\sigma\circ v_{s})f](z_{\rho})$ concentrates as $\rho \to 0$.
  Let $g \colon \R^m \to \R$ be Lipschitz and $h \colon \R^2 \to \R$ be bounded; these represent entries of $\sigma$ and $f$, respectively.
  To show that the variance of $\m{G}_{r_{\rho}}[(g\circ v_{s})h](z_{\rho})$ tends to zero, we control the two-point correlation function
  \begin{equation}
    \label{eq:ksdef}
    k_{s}(x,y)\coloneqq\Cov\big((g \circ v_s)(x), (g \circ v_s)(y)\big),
  \end{equation}
  noting that
  \begin{equation}
    \label{eq:varexpand}
      \Var\big[\m{G}_{r_{\rho}}[(g\circ v_{s})h](z_{\rho})\big] =\int_{\R^2 \times \R^2} G_{r_{\rho}}(z_{\rho}-x)G_{r_{\rho}}(z_{\rho}-y)h(x)h(y)k_{s}(x,y)\d x\ds \add{y}.
  \end{equation}
  Now, for fixed $x \neq y$, \add{Theorems~1.7 and 8.5} of \cite{DG23b} imply that the pair $\big(v_{s}(x),v_{s}(y)\big)$ converges in law as $\rho\to 0$ to a pair of independent random variables.
  By \cref{eq:moment} and the Lipschitz assumption on $g$, we can use uniform integrability to pass to the $\rho\to 0$ limit in \cref{eq:ksdef} and see that $k_{s}(x,y)\to0$ as $\rho\to 0$ for all $x\ne y$.
  Applying the bounded convergence theorem to \cref{eq:varexpand}, we see that
  \begin{equation*}
    \lim_{\rho\downarrow0} \Var\big[\m{G}_{r_{\rho}}[(g\circ v_{s})h](z_{\rho})\big] = 0.
  \end{equation*}
  Letting $g$ and $h$ range over the components of $\sigma$ and $f$ and performing matrix multiplication, the proposition follows.
\end{proof}
\add{
  To close the section, we confirm that the pre-smoothed noise \cref{eq:SHE} satisfies \cref{eq:macro-small}.
  \begin{proposition}
    \label{prop:macroscopic-mild}
    Let $\sigma$ be $L^2$-subcritical and $u^\rho$ solve \cref{eq:SHE} with $u_0 \in L^\infty(\R^2)$ deterministic.
    Then for all $\psi \in \m{C}_c^\infty(\R_+ \times \R^2)$,
    \begin{equation}
      \label{eq:macroscopic-mild}
      \E \Big|\gamma_\rho^{-1} \int_{\R_+ \times \R^2} \psi_t(x) (\m{T}^{\rho}u^\rho - u^\rho)_t(x) \d t \ds x\Big|^2 \to 0 \quad \text{as } \rho \to 0.
    \end{equation}
  \end{proposition}
  \begin{proof}
    This is Proposition~8.8 in~\cite{DG23b} with two changes: we test against $\psi$ and multiply by $\gamma_\rho^{-1}$.
    These changes ``cancel'' because the computation reduces to the spacetime $L^2$-norm of $\m{G}_{t-s} \psi_t$, which is smaller than that of $\m{G}_{t-s} G_\rho$ from \cite{DG23b} by a factor of $\gamma_\rho$.
    So one can show the same quantitative bound as in \cite[Proposition~8.8]{DG23b}, and \cref{eq:macroscopic-mild} follows.
  \end{proof}
  In~\cite{DG23b}, Propositions 8.6 and 8.8 show that solutions $u$ of \cref{eq:SHE} satisfy the hypotheses \zcref[range]{eq:moment,eq:pointwise-small} provided $\sigma$ is $L^2$-subcritical.
  We have just verified \cref{eq:macro-small} as well.
  Hence \cref{thm:maintheorem} follows from \cref{thm:universal}, and it suffices to show the latter.
}

\section{The martingale argument}

Throughout this section, we assume that $\sigma$ is \add{$L^2$-}subcritical and \add{$v$ satisfies \zcref[range]{eq:moment,eq:macro-small}.}
We fix a time $t>0$ and a deterministic test function $\psi\in\m{C}_{\mathrm{c}}^{\infty}\big((0,t)\times\R^{2}\big)$.
\add{%
  Recall the mild operator $\m{T}^\rho$ from \cref{eq:mild}.
  By \cref{eq:macro-small}, it suffices to use $\m{T}^\rho v^\rho$ in place of $u^\rho$ in \cref{eq:mainthm-convergence}.%
}
Let
\begin{equation*}
  \Psi_{s}(x)\coloneqq\int_{s}^{t}\m{G}_{r-s}\psi_{r}(x)\d r,
\end{equation*}
which satisfies the backward forced heat equation
\begin{equation}
  \label{eq:backwardsHE}
  \partial_{s}\Psi_{s}(x)=-\frac{1}{2}\Delta\Psi_{s}(x)-\psi_{s}(x),\qquad \Psi_{t}\equiv0.
\end{equation}
\add{Recalling $\bar{v}_t \coloneqq \m{G}_t v_0$, we define}
\begin{equation*}
  M_{s}\coloneqq\gamma_{\rho}^{-1}\int_{\R^2}(\add{\m{T}}v-\overline{v})_s(x)\Psi_{s}(x)\d x+\gamma_{\rho}^{-1}\int_{0}^{s}\!\!\!\int_{\R^2}(\add{\m{T}} v-\overline{v})_{r}(x)\psi_{r}(x)\d x\ds r.
\end{equation*}
\add{So $M_0 = 0$ and our object of interest is the terminal value}
\begin{equation*}
  M_{t} = \gamma_{\rho}^{-1}\int_{0}^{t}\!\!\!\int_{\R^2} (\add{\m{T}} v-\overline{v})_r(x)\psi_{r}(x) \d x\ds r.
\end{equation*}
\add{Using \cref{eq:mild}, one can check that}
  \begin{equation}
    \label{eq:mild-diff}
    \add{\dn(\m{T} v)_t = \frac{1}{2} \Delta (\m {T} v)_t + \gamma_\rho \m{G}_\rho[\sigma(v_t) \ds W_t].}
  \end{equation}
Then \cref{eq:backwardsHE}, \cref{eq:mild-diff}, Itô's formula, and integration by parts yield
\begin{align}
  \dn M_s &= \gamma_\rho^{-1}\int_{\R^2} \left[\frac{1}{2} \Delta (\add{\m{T}} v - \bar{v})_s \Psi_s - (\add{\m{T}} v - \bar{v})_s\left(\frac{1}{2}\Delta \Psi_s + \psi_s\right) + (\add{\m{T}} v - \bar{v})_s\psi_s\right] \ds x \ds s\nonumber\\
          &\hspace{8cm}+ \int_{\R^2} \Psi_s \m{G}_{\rho}[\sigma(v_s)\d W_{s}] \d x\nonumber\\
          &= \int_{\R^2}\Psi_{s}\m{G}_{\rho}[\sigma(v_s) \d W_{s}] \d x = \int_{\R^2} (\m{G}_{\rho}\Psi_{s}) \sigma(v_s) \d W_{s}.\label{eq:dMs}
\end{align}
\add{Now \cref{eq:moment}} ensures that $\sigma(v_{s})$ has uniformly bounded second moment, so $(M_{s})_{s}$ is a vector-valued martingale with matrix-valued quadratic variation
\begin{equation}
  \label{eq:MQV}
  [M]_{\add{s}}=\int_{0}^{\add{s}}\!\!\!\int_{\R^2} (\m{G}_{\rho}\Psi_{r})(x)^{2} (\sigma^{2} \circ v_{r})(x) \d x\ds r.
\end{equation}

We now define an analogous martingale for the Edwards--Wilkinson solution $\overline{\m{U}}$ of \cref{eq:EW-twonoises-1}.
We consider a test function $\phi\in\m{C}_{\mathrm{c}}^{\infty}\big((0,t)\times\R^{2}\big)$ and define
\begin{equation*}
  \Phi_{s}(x) \coloneqq\int_{s}^{t}\m{G}_{r-s}\phi_{r}(x)\d r
\end{equation*}
as well as
\begin{equation*}
    N_{s} \coloneqq \int_{\R^2} \m{U}_{s} \Psi_{s}\d x + \int_{0}^{s}\!\!\!\int_{\R^2} \m{U}_{r} \psi_{r}\d x\ds r \ \ \text{and}\ \  \bar{N}_s \coloneqq \int_{\R^2} \bar{\m{U}}_{s} \Phi_{s}\d x + \int_{0}^{s}\!\!\!\int_{\R^2} \bar{\m{U}}_{r} \phi_{r}\d x\ds r.
\end{equation*}
By a similar application of Itô's formula, these are martingales \add{with quadratic covariation}
\begin{gather*}
  [N]_{\add{s}} = \int_0^{\add{s}} \!\!\! \int_{\R^2} \Psi_r^2 (J_1^2 \circ \bar{v}_r)\d x \ds r, \quad [\bar{N}]_{\add{s}}= \int_0^{\add{s}} \!\!\! \int_{\R^2} \Phi_r^2 (\Jbar_1^2 \circ \bar{v}_r)\d x \ds r,\\
    [N, \bar{N}]_{\add{s}} = \int_0^{\add{s}} \!\!\! \int_{\R^2} \Psi_r \Phi_r (J_1 \circ \bar{v}_r)(\Jbar_1 \circ \bar{v}_r)\d x \ds r
\end{gather*}
and \add{terminal values}
\begin{equation*}
  N_t = \int_{0}^{t}\!\!\!\int_{\R^2} \m{U}_{r} \psi_{r}\d x\ds r, \quad \bar{N}_t = \int_{0}^{t}\!\!\!\int_{\R^2} \bar{\m{U}}_{r} \phi_{r}\d x\ds r.
\end{equation*}

We note that $[N]_{\add{s}}$, $[\overline{N}]_{\add{s}}$, and $[N,\overline{N}]_{\add{s}}$ are deterministic, which reflects the fact that $N_{\add{s}}$ and $\overline{N}_{\add{s}}$ are jointly Gaussian.
In order to prove \cref{thm:maintheorem}, by the martingale central limit theorem (see, e.g.,~\cite[Theorem VIII.3.11]{JS03}) it suffices to show that $[M]_{\add{s}}\to[N]_{\add{s}}$ and $[M,\overline{N}]_{\add{s}}\to[N,\overline{N}]_{\add{s}}$ in probability as $\rho\to 0$ \add{%
  for all $s \in [0, t]$.
  This is a standard approach appearing, for example, in the proof of Proposition~3.5 in~\cite{NN23}.%
}
These limits constitute the following two propositions.
\begin{proposition}
\add{For all $s \in [0, t]$,} we have $[M^{\add{\rho}}]_{\add{s}}\to[N]_{\add{s}}$ in probability as $\rho\to 0$.
\end{proposition}
\begin{proof}
  We expand
  \begin{equation*}
    \Psi_{r}(x)^{2}= \int_{[s,t]^2 \times \R^4} G_{\tau_{1}-r}(x-y_{1})G_{\tau_{2}-r}(x-y_{2})\psi_{\tau_{1}}(y_{1})\psi_{\tau_{2}}(y_{2}) \d y_1 \ds y_2 \ds \tau_1 \ds \tau_2
  \end{equation*}
  and
  \begin{equation*}
    (\m{G}_{\rho}\Psi_{r})(x)^{2}= \int_{[s,t]^2 \times \R^4} G_{\tau_{1}-r+\rho}(x-y_{1})G_{\tau_{2}-r+\rho}(x-y_{2})\psi_{\tau_{1}}(y_{1})\psi_{\tau_{2}}(y_{2}) \d y_1 \ds y_2 \ds \tau_1 \ds \tau_2.
  \end{equation*}
  We recall that
  \begin{equation*}
    G_{q_{1}}(x-y_{1})G_{q_{2}}(x-y_{2})=G_{q_{1}+q_{2}}(y_{1}-y_{2})G_{(q_{1}^{-1}+q_{2}^{-1})^{-1}}\left(\frac{q_{2}y_{1}+q_{1}y_{2}}{q_{1}+q_{2}}-x\right).
  \end{equation*}
  Using the last two displays in \cref{eq:MQV}, we get
  \begin{equation}
    \begin{aligned}[]
      [M]_{s} & = \int_{\R^2} \dn x \int_{0}^{s} \!\dn r\int_{[r,t]^2} \dn \tau_1 \ds \tau_2 \int_{\R^4} \dn y_1 \ds y_2 \; \psi_{\tau_{1}}(y_{1})\psi_{\tau_{2}}(y_{2})G_{\tau_{1}+\tau_{2}-2r+2\rho}(y_{1}-y_{2})\\
              &\hspace{3mm}\times\m{G}_{[(\tau_{1}-r+\rho)^{-1}+(\tau_{2}-r+\rho)^{-1}]^{-1}}(\sigma^{2}\circ v_{r})\left(\frac{(\tau_{2}-r+\rho)y_{1}+(\tau_{1}-r+\rho)y_{2}}{\tau_{1}+\tau_{2}-2r+2\rho}-x\right).
    \end{aligned}
    \label{eq:MQV-expand}
  \end{equation}
  Similarly,
  \begin{equation}
    \label{eq:NQV-expand}
    \begin{aligned}[]
      [N]_{s} & = \int_{\R^2} \dn x \int_{0}^{s} \!\dn r\int_{[r,t]^2} \dn \tau_1 \ds \tau_2 \int_{\R^4} \dn y_1 \ds y_2 \; \psi_{\tau_{1}}(y_{1})\psi_{\tau_{2}}(y_{2})G_{\tau_{1}+\tau_{2}-2r}(y_{1}-y_{2})\\
              &\hspace{19mm}\times\m{G}_{[(\tau_{1}-r)^{-1}+(\tau_{2}-r)^{-1}]^{-1}}(H_{1}\circ\overline{v}_{r})\left(\frac{(\tau_{2}-r)y_{1}+(\tau_{1}-r)y_{2}}{\tau_{1}+\tau_{2}-2r}-x\right).
    \end{aligned}
  \end{equation}
  By \cref{eq:sigma2concinprb}, the second line of \cref{eq:MQV-expand} converges in probability as $\rho\to0$ to the second line of \cref{eq:NQV-expand}.
  Uniform integrability from \cref{eq:moment} implies that $\E|[M]_{s}-[N]_{s}|_{\mathrm{F}}\to 0$ as $\rho \to 0$, and Markov's inequality completes the proof.
\end{proof}
\begin{proposition}
  \add{For all $s \in [0, t]$,} $[M^{\add{\rho}},\overline{N}]_{\add{s}}\to[N,\overline{N}]_{\add{s}}$ in probability as $\rho\to 0$.
\end{proposition}
\begin{proof}
  Using \cref{eq:dMs} and its analogues for $\bar{N}_t$ and $\overline{N}_{t}$, we compute
  \begin{equation}
    [M,\overline{N}]_{s} = \int_{0}^{s} \!\!\! \int_{\R^2} (\m{G}_{\rho}\Psi_{r})(x)\Phi_{r}(x)(\sigma \circ v_r)(x) (\Jbar_{1} \circ \bar{v}_r)(x) \d x\ds r\label{eq:MNbarQV}
  \end{equation}
  and
  \begin{equation*}
    [N,\overline{N}]_{s} = \int_{0}^{s}\!\!\!\int_{\R^2} \Psi_{r}(x)\Phi_{r}(x) (\Jbar_{1} \circ \bar{v}_r)(x)^2 \d x \ds r.
  \end{equation*}
  We can expand
  \begin{equation*}
    \Psi_{r}(x)\Phi_{r}(x) = \int_{[r,t]^2 \times \R^4} G_{\tau_{1}-r}(x-y_{1})G_{\tau_{2}-r}(x-y_{2})\psi_{\tau_{1}}(y_{1})\phi_{\tau_{2}}(y_{2}) \d y_1 \ds y_2 \ds \tau_1 \ds \tau_2
  \end{equation*}
  and
  \begin{equation*}
    (\m{G}_{\rho}\Psi_{r})(x)\Phi_{r}(x) = \int_{[r,t]^2 \times \R^4} G_{\tau_{1}-r+\rho}(x-y_{1})G_{\tau_{2}-r}(x-y_{2})\psi_{\tau_{1}}(y_{1})\phi_{\tau_{2}}(y_{2}) \d y_1 \ds y_2 \ds \tau_1 \ds \tau_2.
  \end{equation*}
  We can thus develop \cref{eq:MNbarQV} as
  \begin{equation}
    \begin{aligned}[]
      \hspace*{-2mm}[M,&\overline{N}]_{s}= \int_{\R^2} \dn x \int_{0}^{s} \!\dn r\int_{[r,t]^2} \dn \tau_1 \ds \tau_2 \int_{\R^4} \dn y_1 \ds y_2 \; \psi_{\tau_{1}}(y_{1})\psi_{\tau_{2}}(y_{2})G_{\tau_{1}+\tau_{2}-2r+\rho}(y_{1}-y_{2})\\
                            &\hspace{3mm}\times\m{G}_{[(\tau_{1}-r+\rho)^{-1}+(\tau_{2}-r)^{-1}]^{-1}}[(\sigma\circ v_{r})(\Jbar_{1}\circ\overline{v}_{r})]\left(\frac{(\tau_{2}-r)y_{1}+(\tau_{1}-r+\rho)y_{2}}{\tau_{1}+\tau_{2}-2r+\rho}-x\right),
    \end{aligned}
    \label{eq:MNbarexpansion}
  \end{equation}
  and similarly
  \begin{equation}
    \begin{aligned}[]
      [N,\overline{N}]_{s}= \int_{\R^2} &\dn x \int_{0}^{s} \!\dn r\int_{[r,t]^2} \dn \tau_1 \ds \tau_2 \int_{\R^4} \dn y_1 \ds y_2 \; \psi_{\tau_{1}}(y_{1})\psi_{\tau_{2}}(y_{2})G_{\tau_{1}+\tau_{2}-2r}(y_{1}-y_{2})\\
                            &\times\m{G}_{[(\tau_{1}-r+\rho)^{-1}+(\tau_{2}-r)^{-1}]^{-1}}\big[(\Jbar_{1}\circ\overline{v}_{r})^{2}\big]\left(\frac{(\tau_{2}-r)y_{1}+(\tau_{1}-r)y_{2}}{\tau_{1}+\tau_{2}-2r}-x\right).
    \end{aligned}
    \label{eq:NNbarexpansion}
  \end{equation}
  By \cref{eq:sigmaconcinprob} with $f=\Jbar_{1}\circ\overline{v}_{r}$, the second line of \cref{eq:MNbarexpansion} converges to the second line of \cref{eq:NNbarexpansion} in probability as $\rho \to 0$.
  Uniform integrability from \cref{eq:moment} implies that 
  \begin{equation*}
    \E\absb{[M,\overline{N}]_{s}-[N,\overline{N}]_{s}}_{\mathrm{F}}\to0\qquad\text{as }\rho\downarrow0,
  \end{equation*}
  and Markov's inequality completes the proof.
\end{proof}

\providecommand{\bysame}{\leavevmode\hbox to3em{\hrulefill}\thinspace}
\providecommand{\MR}{\relax\ifhmode\unskip\space\fi MR }
\providecommand{\MRhref}[2]{%
  \href{http://www.ams.org/mathscinet-getitem?mr=#1}{#2}
}
\providecommand{\href}[2]{#2}

\begin{acks}
CG was supported by the NSF Mathematical Sciences Postdoctoral Research Fellowship program under grant DMS-2103383.
\end{acks}

\begin{thebibliography}{10}

\bibitem{BC98}
Lorenzo Bertini and Nicoletta Cancrini, \emph{The two-dimensional stochastic
  heat equation: renormalizing a multiplicative noise}, J. Phys. A \textbf{31}
  (1998), no.~2, 615--622. \MR{1629198}

\bibitem{CC22}
Francesco Caravenna and Francesca Cottini, \emph{Gaussian limits for
  subcritical chaos}, Electron. J. Probab. \textbf{27} (2022), Paper No. 81,
  35. \MR{4441143}

\bibitem{CCR23}
Francesco Caravenna, Francesca Cottini, and Maurizia Rossi,
  \emph{Quasi-critical fluctuations for 2d directed polymers}, Ann. Appl. Probab.
  \textbf{35} (2025), no.~4, 2604--2643. \MR{4945087}

\bibitem{CSZ17}
Francesco Caravenna, Rongfeng Sun, and Nikos Zygouras, \emph{Universality in
  marginally relevant disordered systems}, Ann. Appl. Probab. \textbf{27}
  (2017), no.~5, 3050--3112. \MR{3719953}

\bibitem{CSZ19}
\bysame, \emph{On the moments of the {$(2+1)$}-dimen\-sional directed polymer
  and stochastic heat equation in the critical window}, Comm. Math. Phys.
  \textbf{372} (2019), no.~2, 385--440. \MR{4032870}

\bibitem{CSZ20}
\bysame, \emph{The two-dimensional {KPZ} equation in the entire subcritical
  regime}, Ann. Probab. \textbf{48} (2020), no.~3, 1086--1127. \MR{4112709}

\bibitem{CSZ21}
\bysame, \emph{The critical 2d stochastic heat flow}, Invent. Math.
  \textbf{233} (2023), no.~1, 325--460. \MR{4602000}

\bibitem{CN21}
Cl{\'e}ment Cosco and Shuta Nakajima, \emph{Gaussian fluctuations for the
  directed polymer partition function in dimension {$d \geq{}3$} and in the
  whole {$L^2$}-region}, Ann. Inst. H. Poincar{\'e} Probab. Stat. \textbf{57}
(2021), no.~2, 872--889. \MR{4260488}

\bibitem{CNN22}
Cl{\'e}ment Cosco, Shuta Nakajima, and Makoto Nakashima, \emph{Law of large
  numbers and fluctuations in the sub-critical and {$L^2$} regions for {SHE}
  and {KPZ} equation in dimension {$d\geq3$}}, Stochastic Process. Appl.
  \textbf{151} (2022), 127--173. \MR{4441505}

\bibitem{DG23b}
Alexander Dunlap and Cole Graham, \emph{The 2d nonlinear stochastic heat
  equation: pointwise statistics and the decoupling function}, arXiv e-prints
  (2023), \href{https://arxiv.org/abs/2308.11850}{2308.11850}.

\bibitem{DG23a}
\bysame, \emph{Uniqueness and root-{L}ipschitz regularity for a degenerate heat
  equation}, SIAM J. Math. Anal. \textbf{57} (2025), no.~2, 1866--1891. \MR{4894196}

\bibitem{DG22}
Alexander Dunlap and Yu~Gu, \emph{A forward-backward {SDE} from the 2{D}
  nonlinear stochastic heat equation}, Ann. Probab. \textbf{50} (2022), no.~3,
  1204--1253. \MR{4413215}

\bibitem{DGRZ20}
Alexander Dunlap, Yu~Gu, Lenya Ryzhik, and Ofer Zeitouni, \emph{Fluctuations of
  the solutions to the {KPZ} equation in dimensions three and higher}, Probab.
  Theory Related Fields \textbf{176} (2020), no.~3-4, 1217--1258. \MR{4087492}

\bibitem{GHL23}
Luca Gerolla, Martin Hairer, and Xue-Mei Li, \emph{Fluctuations of stochastic
  {PDEs} with long-range correlations}, Ann. Appl. Probab. \textbf{35} (2025),
  no.~2, 1198--1232. \MR{4897759}

\bibitem{Gu20}
Yu~Gu, \emph{Gaussian fluctuations from the 2{D} {KPZ} equation}, Stoch.
  Partial Differ. Equ. Anal. Comp. \textbf{8} (2020), no.~1, 150--185.
  \MR{4058958}

\bibitem{GL20}
Yu~Gu and Jiawei Li, \emph{Fluctuations of a nonlinear stochastic heat equation
  in dimensions three and higher}, SIAM J. Math. Anal. \textbf{52} (2020),
  no.~6, 5422--5440. \MR{4169750}

\bibitem{GQT21}
Yu~Gu, Jeremy Quastel, and Li-Cheng Tsai, \emph{Moments of the 2{D} {SHE} at
  criticality}, Probab. Math. Phys. \textbf{2} (2021), no.~1, 179--219.
  \MR{4404819}

\bibitem{GRZ18}
Yu~Gu, Lenya Ryzhik, and Ofer Zeitouni, \emph{The {E}dwards-{W}ilkinson limit
  of the random heat equation in dimensions three and higher}, Comm. Math.
  Physics. \textbf{363} (2018), no.~2, 351--388. \MR{3851818}

\bibitem{JS03}
Jean Jacod and Albert~N. Shiryaev, \emph{Limit theorems for stochastic
  processes}, second ed., Grundlehren Math. Wiss.
  [Fundamental Principles of Mathematical Sciences] \textbf{288}, Springer-Verlag,
  Berlin, 2003. \MR{1943877}

\bibitem{Kot24}
Sotirios Kotitsas, \emph{The heat equation with time-correlated random
  potential in d=2: Edwards-wilkinson fluctuations}, arXiv e-prints (2024),
  \href{https://arxiv.org/abs/2405.01519}{2405.01519}.

\bibitem{LZ22}
Dimitris Lygkonis and Nikos Zygouras, \emph{Edwards-{W}ilkinson fluctuations
  for the directed polymer in the full {$L^2$}-regime for dimensions
  {$d\geq{}3$}}, Ann. Inst. H. Poincar{\'e} Probab. Stat. \textbf{58} (2022),
  no.~1, 65--104. \MR{4374673}

\bibitem{MU18}
Jacques Magnen and J{\'e}r{\'e}mie Unterberger, \emph{The scaling limit of the
  {KPZ} equation in space dimension 3 and higher}, J. Stat. Phys. \textbf{171}
  (2018), no.~4, 543--598. \MR{3790153}

\bibitem{NN23}
Shuta Nakajima and Makoto Nakashima, \emph{Fluctuations of two-dimensional
  stochastic heat equation and {KPZ} equation in subcritical regime for general
  initial conditions}, Electron. J. Probab. \textbf{28} (2023), Paper No. 1,
  38. \MR{4529085}

\bibitem{Tao22}
Ran Tao, \emph{Gaussian fluctuations of a nonlinear stochastic heat equation in
  dimension two}, Stoch. Partial Differ. Equ. Anal. Comput. \textbf{12} (2024),
  no.~1, 220--246. \MR{4709542}

\bibitem{Zyg24}
Nikos Zygouras, \emph{Directed polymers in a random environment: a review of
  the phase transitions}, Stochastic Processes Appl. \textbf{177} (2024), Paper No. 104431.

\end{thebibliography}
\end{document}